\documentclass[11pt,reqno]{amsart}

\usepackage[utf8]{inputenc}

\usepackage{lineno}
\usepackage{amsaddr}
\usepackage{a4wide}
\usepackage{hyperref}
\usepackage{mathtools}
\usepackage{subcaption}
\usepackage[all]{xy}

\usepackage{mathrsfs}

\usepackage{todonotes}

\usepackage{tikz}
\usetikzlibrary{arrows}

\usepackage{marginnote}
\usepackage{amsmath,amsthm,latexsym,xspace,amscd,amssymb,xy}               % AmSLaTeX
\usepackage{color}
\usepackage{enumerate}

\usepackage[backend=biber,defernumbers=false,firstinits=true, url=false]{biblatex}
\DeclareUnicodeCharacter{2217}{~}

\mathchardef\mhyphen="2D

\newtheorem{theorem}{Theorem}[section]
\newtheorem{lemma}[theorem]{Lemma}
\newtheorem{corollary}[theorem]{Corollary}
\newtheorem{proposition}[theorem]{Proposition}

\theoremstyle{definition}
\newtheorem{definition}[theorem]{Definition}
\newtheorem{example}[theorem]{Example}
\newtheorem{remark}[theorem]{Remark}

\author{Daniel Lännström}
\address{Department of Mathematics and Natural Sciences,
Blekinge Institute of Technology,
SE-37179 Karlskrona, Sweden}

\email{{\scriptsize daniel.lannstrom@bth.se}}
\date{\today}



\title[A characterization of graded von Neumann regular rings]{A characterization of graded von Neumann regular rings with applications to Leavitt path algebras}
\begin{document}
\maketitle

\begin{abstract}
We provide a characterization of graded von Neumann regular rings involving the recently introduced class of nearly epsilon-strongly graded rings. As our main application, we generalize Hazrat's result that Leavitt path algebras over fields are graded von Neumann regular. More precisely, we show that a Leavitt path algebra $L_R(E)$ with coefficients in a unital ring $R$ is graded von Neumann regular if and only if $R$ is von Neumann regular. We also prove that both Leavitt path algebras and corner skew Laurent polynomial rings over von Neumann regular rings are semiprimitive and semiprime. Thereby, we generalize a result by Abrams and Aranda Pino on the semiprimitivity of Leavitt path algebras over fields.
\end{abstract}

\section{Introduction}

 An associative ring $R$ is called \emph{von Neumann regular} if $a \in a Ra$ holds for every $a \in R$. This type of ring was first considered by von Neumann in the study of operator algebras and has since then been extensively studied (see e.g. Goodearl's monograph \cite{goodearl1991neumann}). There are several well-known equivalent statements for a unital ring to be von Neumann regular. Notably, a unital ring $R$ is von Neumann regular if and only if every finitely generated left  (right) ideal of $R$ is generated by an idempotent. Examples of von Neumann regular rings are plentiful. For instance, any field is von Neumann regular. On the other hand, the ring of integers $\mathbb{Z}$ (or any non-field integral domain) is not von Neumann regular. 
 
Let $G$ be a group with neutral element $e$ and let $S = \bigoplus_{g \in G} S_g$ be a $G$-graded ring (see Section \ref{sub:graded}). The ring $S$ is called \emph{graded von Neumann regular} if, for every $g \in G$ and $a \in S_g$, the relation $a \in a S a$ holds. In the case of unital rings, this notation was introduced by Năstăsescu and van Oystaeyen \cite{nastasescu1982graded} and has been further studied in \cite{Balaba2012,hazrat2016graded,Oystaeyen1984,yahya1997note}. In this article, we will continue the study initiated by Hazrat \cite{hazrat2014leavitt} of non-unital graded von Neumann regular rings. There are results on the graded ideal structure of graded von Neumann regular rings (see \cite[Prop. 1-2]{hazrat2014leavitt} and Proposition \ref{prop:hazrat_char}) which make this class of rings interesting to us.  For the special class of unital strongly group graded rings, the following result highlights a connection between von Neumann regularity and graded von Neumann regularity:
 
\begin{theorem}[Năstăsescu and van Oystaeyen {\cite[Cor. C.I.1.5.3]{nastasescu1982graded}}]
Let $S=\bigoplus_{g \in G} S_g$ be a unital strongly $G$-graded ring. Then $S$ is graded von Neumann regular if and only if $S_e$ is von Neumann regular. 
\label{thm:hazrat}
\end{theorem}

Theorem \ref{thm:hazrat} was originally proved by Năstăsescu and van Oystaeyen using Dade's theorem (see e.g. \cite[Thm. 1.5.1] {hazrat2016graded}). An elementwise proof of Theorem \ref{thm:hazrat} was later given by Yahya \cite[Thm. 3]{yahya1997note}. In this article, we recover Theorem \ref{thm:hazrat} as a special case of our characterization of general graded von Neumann regular rings (see Theorem \ref{thm:char}). 

%Moreover, 

\smallskip

The notion of an \emph{epsilon-strongly graded ring} (see Definition \ref{def:epsilon}) was introduced by Nystedt, Öinert and Pinedo \cite{nystedt2016epsilon} as a generalization of unital strongly graded rings. This class of graded rings includes: unital partial crossed products (see \cite[pg. 2]{nystedt2016epsilon}), corner skew Laurent polynomial rings (see \cite[Thm. 8.1]{lannstrom2019graded}) and Leavitt path algebras of finite graphs (see \cite[Thm. 1.2]{nystedt2017epsilon}). The further generalization to \emph{nearly epsilon-strongly graded rings} (see Definition \ref{def:nearly}) was recently introduced by Nystedt and Öinert \cite{nystedt2017epsilon}. In this article, we study the relation between these two recently introduced classes of graded rings and the classical notion of graded von Neumann regular rings. Our main result is the following characterization:
\begin{theorem}
Let $S=\bigoplus_{g \in G} S_g$ be a $G$-graded ring. Then $S$ is graded von Neumann regular if and only if $S$ is nearly epsilon-strongly $G$-graded and
$S_e$ is von Neumann regular.
%\end{item}
%\begin{item}
%for every $g \in G$ and $x \in S_g$, the left $S_e$-ideal $S_{g^{-1}}x$ is finitely generated;
%\end{item}
%\end{enumerate}
\label{thm:char}
\end{theorem}
%\begin{remark}
%Condition (c) is of a technical nature and is used in the proof of Theorem \ref{thm:char}. It is not clear to the author if there are any rings satisfying (a) and (b) but not (c). 
%\end{remark}

By applying Theorem \ref{thm:char}, we generalize Theorem \ref{thm:hazrat} to the class of epsilon-strongly graded rings (see Corollary \ref{cor:epsilon}). This allows us to characterize when unital partial crossed products (see Corollary \ref{cor:partial_crossed_products}), corner skew Laurent polynomial rings (see Corollary \ref{cor:corner}) and Leavitt path algebras over unital rings (see Theorem \ref{thm:big}) are graded von Neumann regular.

\subsection{Applications to Leavitt path algebras}

Given a directed graph $E$ and a field $K$, the \emph{Leavitt path algebra} $L_K(E)$ is an associative $\mathbb{Z}$-graded $K$-algebra. These algebras were introduced by Ara, Moreno and Pardo \cite{ara2007nonstable} and independently by Abrams and Aranda Pino \cite{abrams2005leavitt}. In the 15 years since their introduction, Leavitt path algebras have found  applications in general ring theory and matured into a research topic of their own (see e.g. \cite{abrams2017leavitt}). There are many results in the literature relating properties of the graph $E$ with algebraic properties of $L_K(E)$. For instance, $L_K(E)$ is von Neumann regular if and only if $E$ is an acyclic directed graph (see \cite{abrams2010regularity}). For graded von Neumann regularity, Hazrat has obtained the following result:

\begin{theorem}[Hazrat {\cite{hazrat2014leavitt}}]
Let $K$ be a field and let $E$ be a directed graph. Then the $\mathbb{Z}$-graded Leavitt path algebra $L_K(E)$ is graded von Neumann regular.
\label{thm:hazrat_lpa}
\end{theorem}
%Hazrat \cite{hazrat2014leavitt} outlines  an approach where  Theorem \ref{thm:hazrat} is used to prove Theorem \ref{thm:hazrat_lpa} for the proper subclass of strongly graded Leavitt path algebras. For the general case, however, he uses  a more involved technique based on corner skew Laurent polynomial rings. In this article, we employ Theorem \ref{thm:char} together with Hazrat's original proof idea to obtain the following  generalization of Theorem \ref{thm:hazrat_lpa}: 

%However, note that the trivial case of the empty graph is  degenerate and needs to be excluded (see Example \ref{ex:2}).
 
Tomforde \cite{tomforde2011leavitt} introduced Leavitt path algebras over commutative unital rings and proved that many results carry over to his generalized setting. Leavitt path algebras over $\mathbb{Z}$ were considered by Johansen and S{\o}rensen  \cite{johansen2016cuntz} in connection to the classification program of Leavitt path algebras. In this article, we follow Hazrat \cite{hazrat2013graded} and consider Leavitt path algebras $L_R(E)$ where $R$ is a general, possibly non-commutative associative ring. We want to relate algebraic properties of the ring $R$ to algebraic properties of $L_R(E)$.  In that vein, we will establish the following generalization of Theorem \ref{thm:hazrat_lpa}:

\begin{theorem}
Let $R$ be a unital ring and let $E$ be a directed graph. Then the $\mathbb{Z}$-graded Leavitt path algebra $L_R(E)$ is graded von Neumann regular if and only if $R$ is von Neumann regular.
\label{thm:big}
\end{theorem}
\begin{remark}
The statement of Theorem \ref{thm:big} does not hold if $E$ is the null graph, i.e. the graph without any vertices or edges (see Remark \ref{ex:2}).
\end{remark}

%As Hazrat \cite{hazrat2014leavitt} points out, the conclusion of Theorem \ref{thm:hazrat_lpa} follows directly from Theorem \ref{thm:hazrat} for the class of strongly $\mathbb{Z}$-graded Leavitt path algebras. 

Hazrat \cite{hazrat2014leavitt} outlines  an approach where  Theorem \ref{thm:hazrat} is used to prove Theorem \ref{thm:hazrat_lpa} for the proper subclass of strongly graded Leavitt path algebras. For the general case, however, he uses  a more involved technique based on corner skew Laurent polynomial rings. In this article, we employ Theorem \ref{thm:char} together with Hazrat's original proof idea to establish Theorem \ref{thm:big}. 

%Since $\mathbb{Z}$ is not von Neumann regular while $\mathbb{C}$ is von Neumann regular, Theorem \ref{thm:big}, algebraically distinguishes between $L_\mathbb{C}(E)$ and $L_\mathbb{Z}(E)$ (see Remark \ref{rem:3}).

\smallskip

The rest of this article is organized as follows:

\smallskip

In Section \ref{sec:prelim}, we recall some preliminaries on non-unital von Neumann regular rings (Section \ref{sub:von}), group graded rings (Section \ref{sub:graded}), epsilon-strongly and nearly epsilon-strongly graded rings (Section \ref{sub:nearly}) and direct limits of graded rings (Section \ref{sub:limits}).

In Section \ref{sec:main} and Section \ref{sec:applications}, we prove Theorem \ref{thm:char} respectively Theorem \ref{thm:big}.

In Section \ref{sec:semiprime}, we show that a Leavitt path algebra over a von Neumann regular ring is both semiprimitive and semiprime (see Corollary \ref{cor:semiprimitive}). Our result generalizes a well-known result by Abrams and Aranda Pino \cite{abrams2008leavitt} for Leavitt path algebras over fields.

In Section \ref{sec:applications2}, we apply our results to unital partial crossed products (Corollary \ref{cor:partial_crossed_products}) and corner skew Laurent polynomial rings (Corollary \ref{cor:corner} and Corollary \ref{cor:semi2}).

\section{Preliminaries}
\label{sec:prelim}
Throughout this article, all rings are assumed to be associative but not necessarily unital.
\subsection{Non-unital von Neumann regular rings}
\label{sub:von}
 A ring $R$ is called \emph{s-unital} if $x \in xR \cap Rx$ for every $x  \in R$. Equivalently, a ring $R$ is s-unital if, for every $x \in R$, there exist some $e, e' \in R$ such that $x = ex = xe'$. A ring is called \emph{unital} if it is equipped with a non-zero multiplicative identity element. A subset $E$ of $R$ is called a \emph{set of local units} for $R$ if $E$ consists of commuting idempotents such that for every $x \in R$ there exists some $e \in E$ such that $x = ex = xe$. Note that a ring with a set of local units is s-unital. For more details about s-unital rings and rings with local units, we refer the reader to the survey article \cite{nystedt2018unital}.
 
  A ring $R$ is called \emph{von Neumann regular} if for every $x \in R$ there is some $ y \in R$ such that $x = x y x$. In fact, every von Neumann regular ring is s-unital:

\begin{proposition}(cf. \cite[Prop. 20]{nystedt2018unital})
Let $R$ be a ring. If $R$ is von Neumann regular, then $R$ is s-unital.
\label{prop:s-unital1}
\end{proposition} 
\begin{proof}
Take an arbitrary $x \in R$. Then there exists some $y \in R$ such that $x=xyx$. Letting $e := xy$ and $e' := yx$, we see that $x = xyx = ex = x e'$ and hence $R$ is s-unital.
\end{proof}
%\begin{remark}
%Nystedt (see \cite[Prop. 20]{nystedt2018unital}) showed that if $R$ is von Neumann regular, then $R$ is \emph{locally unital} (see \cite[Def. 16]{nystedt2018unital}). Since locally unital rings are s-unital, his result is stronger than Proposition \ref{prop:s-unital1}.
%\end{remark}
 
One of the famous classical characterizations of von Neumann regularity for unital rings generalizes to s-unital rings verbatim. We include parts of the proof for the convenience  of the reader:

\begin{proposition}
Let $R$ be an s-unital ring. Then the following assertions are equivalent:
\begin{enumerate}[(a)]
\begin{item}
$R$ is von Neumann regular;
\end{item}
\begin{item}
every principal right (left) ideal of $R$ is generated by an idempotent;
\end{item}
\begin{item}
every finitely generated right (left) ideal of $R$ is generated by an idempotent.
\end{item}
\end{enumerate}
\label{prop:von1}
\end{proposition}
\begin{proof}
$(a) \Rightarrow (b):$ Take an arbitrary $x \in R$ and let $y \in R$ such that $x=xyx$. Consider the right ideal $xR$. Then $xy \in R$ is an idempotent such that $xR = xyR$. 

$(b) \Rightarrow (c):$ See \cite[Thm. 1.1]{goodearl1991neumann}.

$(c) \Rightarrow (a):$ Take an arbitrary $x \in R$. Then $xR =fR$ for some idempotent $f \in R$. Since $R$ is s-unital, $f \in fR=xR$ and hence $f = xy$ for some $y \in R$. Similarly, $x \in xR = fR$ and hence $x = f r$ for some $r \in R$. Then $x = fr = f^2 r = f(fr)= f x = xyx$. 
\end{proof}
%\begin{remark}
%Hazrat established graded version of Lemma \ref{lem:von1} for certain graded von Neumann regular rings (see \cite[Prop. 1]{hazrat2014leavitt}).
%\end{remark}

\subsection{Group graded rings}
\label{sub:graded}
Let $G$ be a group with neutral element $e$. A \emph{$G$-grading} of a ring $S$ is a collection $\{ S_g \}_{g \in G}$ of additive subsets of $S$ such that  $S = \bigoplus_{g \in G} S_g$ and $S_g S_h \subseteq S_{g h}$ for all $g, h\in G$.  The ring $S$ is then called \emph{$G$-graded}.  If the stronger condition $S_g S_h = S_{gh}$ holds for all $g,h \in G$, then the grading is called \emph{strong} and $S$ is called \emph{strongly $G$-graded}. The subsets $S_g$ are called the \emph{homogeneous components} of $S$.  The \emph{principal component}, $S_e$, is a subring of $S$.  A \emph{homogeneous element} $s \in S$ is an element such that $s\in S_g$ for some $g \in G$. Every element of $S$ decomposes uniquely into a sum of homogeneous elements.  A left/right/two-sided ideal $I$ of $S$ is called a left/right/two-sided \emph{graded ideal} of $S$ if $I = \bigoplus_{g \in G} (I \cap S_g)$. 

Recall that a $G$-graded ring $S$ is \emph{graded von Neumann regular} if and only if $a \in aSa$ for every homogeneous element $a \in S$. However, it is possible to make this condition more precise. The following result is well-known, but we have chosen to include a proof for the convenience of the reader.

\begin{proposition}
A $G$-graded ring $S$ is graded von Neumann regular if and only if, for every homogeneous $a\in S_g$ there is some homogeneous $b \in S_{g^{-1}}$ such that $a=aba$.
\label{prop:ny}
\end{proposition}
\begin{proof}
The `if' direction is clear. Conversely, take an arbitrary homogeneous element $a \in S_g$. By assumption, there exists some $b \in S$ such that $a = a ba $. Let $b = \sum_{h \in G} b_h$ be the decomposition of $b$. Note that $a = a ba = \sum_{h \in G} a b_h a$. Since the decomposition is unique, it follows that $b = b_{g^{-1}} \in S_{g^{-1}}$.
\end{proof}

 A $G$-graded ring that is von Neumann regular is graded von Neumann regular. On the other hand, the following is an example of a graded von Neumann regular ring which is not von Neumann regular:
\begin{example}
Let $K$ be a field and consider the Laurent polynomial ring $K[x,x^{-1}]$ with its canonical $\mathbb{Z}$-grading, i.e. $K[x, x^{-1}] = \bigoplus_{i \in \mathbb{Z}} K x^i$. A routine check shows that this gives a strong $\mathbb{Z}$-grading. Since $K$ is von Neumann regular, it follows by Theorem \ref{thm:hazrat} that $K[x, x^{-1}]$ is graded von Neumann regular. On the other hand, $K[x, x^{-1}]$ is an integral domain which is not a field. Therefore, $K[x, x^{-1}]$ is not von Neumann regular.
\end{example}

Let $S=\bigoplus_{g \in G} S_g$ be a $G$-graded ring and suppose that $E$ is a set of local units for $S$. If $E$ consists of homogeneous idempotents, then $E$ is called \emph{a set of homogeneous local units}. Notably, every Leavitt path algebra has a set of homogeneous local units (see Section \ref{sec:semiprime}). Moreover, Hazrat has established the following `graded version' of Proposition \ref{prop:von1}:%characterization of graded von Neumann regularity for graded rings with homogeneous local units (cf. Proposition \ref{prop:von1}):

\begin{proposition}(\cite[Prop. 1]{hazrat2014leavitt})
Let $S=\bigoplus_{g \in G} S_g$ be a $G$-graded ring. Suppose that $S$ has a set of homogeneous local units. Then the following three assertions are equivalent:
\begin{enumerate}[(a)]
\begin{item}
$S$ is graded von Neumann regular;
\end{item}
\begin{item}
every principal right (left) graded ideal of $S$ is generated by a homogeneous idempotent;
\end{item}
\begin{item}
every finitely generated right (left) graded ideal of $S$ is generated by a homogeneous idempotent.
\end{item}
\end{enumerate}
\label{prop:hazrat_char}
\end{proposition}

%Hazrat also established the following properties for graded von Neumann regular rings:
%\begin{proposition}(\cite[Prop. 1]{hazrat2014leavitt})
%Let $S$ be a $G$-graded ring and suppose that $S$ is graded von Neumann regular. Then the following assertions hold:
%\begin{enumerate}[(a)]
%\begin{item}
%Every graded right (left) ideal of $S$ is idempotent;
%\end{item}
%\begin{item}
%Every graded right (left) ideal of $S$ is semiprime;
%\end{item}
%\begin{item}
%Every finite generated right (left) ideal of $S$ is a projective module.
%\end{item}
%\end{enumerate}
%\end{proposition}

\subsection{Nearly epsilon-strongly graded rings}
\label{sub:nearly}
Next, we recall two special types of group graded rings generalizing the classical notion of unital strongly group graded rings. Nystedt, Öinert and Pinedo \cite{nystedt2016epsilon} recently introduced the class of epsilon-strongly $G$-graded rings:

\begin{definition}(\cite[Prop. 7(iii)]{nystedt2016epsilon})
Let $S=\bigoplus_{g \in G} S_g$ be a $G$-graded ring. Suppose that for every $ g \in G$ there is an element $\epsilon_g \in S_g S_{g^{-1}}$ such that for every $ s\in S_g$ the relations $\epsilon_g s = s = s \epsilon_{g^{-1}}$ hold. Then $S$ is called \emph{epsilon-strongly $G$-graded}.
\label{def:epsilon}
\end{definition}

%The following families of rings have been shown to be epsilon-strongly graded:
%\begin{enumerate}[(a)]
%\begin{item}
%unital partial crossed products (see \cite[pg. 2]{nystedt2016epsilon}),
%\end{item}
%\begin{item}
%Leavitt path algebras of finite graphs (see \cite[Thm. 28]{nystedt2017epsilon}),
%\end{item}
%\begin{item}
%corner skew Laurent polynomial rings (see \cite[Thm. 8.1]{lannstrom2019graded}). 
%\end{item}
%\end{enumerate}

\begin{remark}
Let $S = \bigoplus_{g \in G}S_g$ be a $G$-graded ring. We make the following two remarks:

\begin{enumerate}[(a)]
\begin{item}
If $S$ is a unital strongly  $G$-graded ring, then $1 \in S_g S_{g_{-1}}$ for every $g \in G$ (see e.g. \cite[Prop. 1.1.1]{nastasescu1982graded}). In this case, $S$ is epsilon-strongly $G$-graded with $\epsilon_g := 1$ for every $g \in G$. This proves that unital strongly $G$-graded rings are epsilon-strongly $G$-graded. 
\end{item}
\begin{item}
If $S$ is an epsilon-strongly $G$-graded ring, then $S$ is a unital ring (see \cite[Prop. 3.8]{induced2018}). In other words, only unital rings admit epsilon-strong $G$-gradings.
\end{item}
\end{enumerate}
\label{rem:0}
\end{remark}

The following is an example of a $\mathbb{Z}$-graded ring that is epsilon-strongly $\mathbb{Z}$-graded but not strongly $\mathbb{Z}$-graded:

\begin{example}
Let $R$ be a unital ring and consider the following $\mathbb{Z}$-grading of the full matrix ring $M_2(R)$: 
\begin{equation*}
(M_2(R))_0 := \begin{pmatrix}
R & 0 \\
0 & R 
\end{pmatrix}, \quad (M_2(R))_{-1} := \begin{pmatrix}
0 & 0 \\
R & 0
\end{pmatrix}, \quad
(M_2(R))_{1} := \begin{pmatrix}
0 & R \\
0 & 0
\end{pmatrix},
\end{equation*}
and $(M_2(R))_i := \left\{ \bigl( \begin{smallmatrix}0 & 0\\ 0 & 0\end{smallmatrix}\bigr) \right\}$ for $|i| > 1$. Note that,
\begin{equation*}
(M_2(R))_1 (M_2(R))_{-1} = \begin{pmatrix}
R & 0 \\
0 & 0 
\end{pmatrix}, \quad (M_2(R))_{-1} (M_2(R))_1 = \begin{pmatrix}
0 & 0 \\
0 & R
\end{pmatrix}.
\end{equation*}
A routine check shows that $M_2(R)$ is epsilon-strongly $\mathbb{Z}$-graded with,
\begin{equation*}
\epsilon_1 := \begin{pmatrix}
1_R & 0 \\
0 & 0 
\end{pmatrix}, \quad 
\epsilon_{-1} := \begin{pmatrix}
0 & 0 \\
0 & 1_R
\end{pmatrix}, \quad
\epsilon_0 := \begin{pmatrix}
1_R & 0 \\
0 & 1_R
\end{pmatrix} \text{ and }
\epsilon_i := \begin{pmatrix}
0 & 0 \\
0 & 0
\end{pmatrix} \text{ for } |i| > 1.
\end{equation*}
However, since an epsilon-strong $\mathbb{Z}$-grading is strong if and only if $\epsilon_i = 1$ for every $i \in \mathbb{Z}$ (see \cite[Prop. 3.2]{nystedt2017epsilon}), it follows that the above $\mathbb{Z}$-grading of $M_2(R)$ is not strong. %Indeed, by \cite[Prop 1.1.1(3)]{nastasescu2004methods}, $M_2(R)$ is strongly $\mathbb{Z}$-graded if and only if $1_{M_2(R)} \in (M_2(R))_i (M_2(R))_{-i}$ for every $i \in \mathbb{Z}$. But $1_{M_2(R)} \not \in (M_2(R))_i (M_2(R))_{-i} = \{ 0 \}$ for $|i| > 1$.
%Moreover, a routine check shows that (\ref{eq:d1}) and (\ref{eq:d2}) hold. Thus, $M_2(R)$ is epsilon-strongly  $\mathbb{Z}$-graded by Definition \ref{def:epsilon-strongly}. Furthermore, since $M_2(R) \cong_{\text{gr}} L_R(A_2)$ (see Example \ref{ex:420}), this is an example of a Leavitt path algebra such that the canonical $\mathbb{Z}$-grading is epsilon-strong but not strong.
\label{ex:54}
\end{example}

Crucial to our investigation, Nystedt and Öinert have shown that a Leavitt path algebra associated to a finite directed graph is epsilon-strongly $\mathbb{Z}$-graded (see \cite[Thm. 1.2]{nystedt2017epsilon}). Seeking to generalize their result to include any Leavitt path algebra (i.e. possibly non-finite graphs), they introduced nearly epsilon-strongly graded rings. %These can be seen as the `s-unital generalization' of epsilon-strongly graded rings.

\begin{definition}(\cite[Prop. 3.3]{nystedt2017epsilon})
Let $S=\bigoplus_{g \in G} S_g$ be a $G$-graded ring. Suppose that for every $ g \in G$ and $s \in S_g$ there are elements $\epsilon_g(s) \in S_g S_{g^{-1}}, \epsilon_g'(s) \in S_{g^{-1}} S_g$  such that the relations $\epsilon_g(s) s = s = s \epsilon_g'(s)$ hold. Then $S$ is called \emph{nearly epsilon-strongly $G$-graded}.
\label{def:nearly}
\end{definition}

Every Leavitt path algebra is indeed nearly epsilon-strongly $\mathbb{Z}$-graded (see \cite[Thm. 1.3]{nystedt2017epsilon}). The following is a trivial example of a nearly epsilon-strongly $G$-graded ring:
\begin{example}
Let $G$ be a group and let $R$ be an s-unital ring that is not unital (for instance, let $R:=C_c(\mathbb{R})$ with pointwise multiplication). Put $R_e := R$ and $R_g := \{ 0 \}$ for every $g \ne e$. This gives a $G$-grading of $R$ called \emph{the trivial $G$-grading}. For every $x \in R$ there are some $e,e' \in R$ such that $x=ex=xe'$. Letting $\epsilon_e(x) := e \in R = R^2 = R_e R_e$ and $\epsilon_e'(x)= e' \in R = R^2 = R_e R_e$ in Definition \ref{def:nearly}, we see that $R$ is nearly epsilon-strongly $G$-graded. On the other hand, since $R$ is not unital, it follows from  Remark \ref{rem:0}(b) that $R$ cannot be epsilon-strongly $G$-graded.
\end{example}

%The following proposition will be needed later:
%
%\begin{proposition}
%Let $S=\bigoplus_{g \in G} S_g$ be a $G$-graded ring. If $S$ is nearly epsilon-strongly $G$-graded, then $S_e$ is an s-unital ring. 
%\label{prop:s-unital1}
%\end{proposition}
%\begin{proof}
%Assume that $S$ is nearly epsilon-strongly $G$-graded. Then, in particular, $S$ is symmetrically $G$-graded (see \cite[Prop. 11]{nystedt2017epsilon}). By \cite[pg. 8]{clark2018generalized}, we have that $S_e$ is an idempotent ring. By \cite[Prop. 11]{nystedt2017epsilon}, the ring $S_g S_{g^{-1}}$ is s-unital for every $g \in G$. Then $S_e^2 = S_e$ is s-unital.
%\end{proof}

\subsection{Direct limits in the category of graded rings}
\label{sub:limits}
We will recall some properties of the category of group graded rings. Let $S=\bigoplus_{g \in G} S_g$ and $T=\bigoplus_{g \in G} T_g$ be two $G$-graded rings. A ring homomorphism $\phi \colon S \to T$ is called \emph{graded} if $\phi(S_g) \subseteq T_g$ for every $g \in G$. If $\phi \colon S \xrightarrow{\sim} T$ is a graded ring isomorphism, then we write $S \cong_{\text{gr}} T$ and say that $S$ and $T$ are \emph{graded isomorphic}. Note that two graded isomorphic rings are also isomorphic but the reverse implication does not hold in general.  If $S \cong_{\text{gr}} T$, then $S$ is graded von Neumann regular if and only if $T$ is graded  von Neumann regular.

The category of $G$-graded rings will be denoted by $G \mhyphen \textrm{RING}$. The objects of this category are pairs $(S, \{S_g \}_{g \in G})$ where $S$ is a ring and $\{ S_g \}_{g \in G}$ is a $G$-grading of $S$. The morphisms of $G \mhyphen \textrm{RING}$ are the $G$-graded ring homomorphisms. 
Next, we consider direct limits in $G \mhyphen \textrm{RING}$. Let $\{ A_i \mid i \in I \}$ be a directed system of $G$-graded rings. For every $i \in I$, we have that $A_i = \bigoplus_{g \in G} (A_i)_g$. Recall (see \cite[II, §11.3, Rem. 3]{bourbaki1998algebra}) that $B = \varinjlim_{i \in I} A_i$ is a $G$-graded ring with homogeneous components $B_g = \varinjlim_{i \in I} (A_i)_g.$ In other words, the category $G \mhyphen \textrm{RING}$ has arbitrary direct limits. The following lemma is a graded version of a well-known result (see \cite[Prop. 5.2.14]{berrick2000categories}). We include a proof for the convenience of the reader.

\begin{lemma}
Let $\{ A_i \mid i \in I \}$ be a directed system of $G$-graded rings. Suppose that $A_i$ is graded von Neumann regular for every $i \in I$. Then $B=\varinjlim_{i} A_i$ is graded von Neumann regular.
\label{lem:direct_lim}
\end{lemma}
\begin{proof}
Let $(B=\varinjlim_{i} A_i, \phi_i)$ be the direct limit of $\{ A_i \mid i \in I \}$. Recall that the canonical functions $\phi_i \colon A_i \to B=\varinjlim_{i} A_i$ are graded ring homomorphisms. Take an arbitrary $g \in G$ and $b_g \in B_g = \varinjlim_{i \in I} (A_i)_g$. Then, $b_g = \phi_k(a_g)$ for some $ a_g \in (A_k)_g$ and $k \in I$. Since $A_k$ is graded von Neumann regular by assumption, it follows that there is some $s \in A_k$ such that $a_g = a_g s a_g$. Applying $\phi_k$ to both sides yields, $b_g = b_g \phi_k(s) b_g$ for $\phi_k(s) \in B$. Thus, $B$ is graded von Neumann regular.
\end{proof}

\section{Main result}
\label{sec:main} %a characterization of graded von Neumann regular rings that generalizes Theorem \ref{thm:hazrat}
In this section, we prove our main result: Theorem \ref{thm:char}. We first show that there are $G$-graded rings $S$ such that $S_e$ is von Neumann regular while $S$ is not graded von Neumann regular. 
\begin{example}
Let $R$ be a von Neumann regular ring (e.g. a field) and consider the polynomial ring $R[x] = \bigoplus_{i \geq 0} R x^i$. By putting $S_i := R x^i$ for $i \geq 0$ and $S_i := \{ 0 \}$ for $i <0$, we get a $\mathbb{Z}$-grading of $R[x]$. Note that $x^2 \not \in (x^2) R[x] (x^2)$. Hence, $R[x]$ is not graded von Neumann regular. This example shows that the conclusion of Theorem \ref{thm:hazrat} does not hold for a general group graded ring. 
\label{ex:3a}
\end{example}

We now consider necessary conditions for a ring to be graded von Neumann regular. The following result is well-known and follows from Proposition \ref{prop:ny}:

\begin{lemma}
Let $S= \bigoplus_{g \in G}S_g$ be a $G$-graded ring. If $S$ is graded von Neumann regular, then $S_e$ is von Neumann regular. 
\label{lem:1}
\end{lemma}
%\begin{proof}
%Take an arbitrary element $a \in S_e$. By the assumption that $S$ is graded von Neumann regular, there exists some $b \in S$ such that $a = a ba $. Let $b = \sum_{g \in G} b_g$ be the decomposition of $b$. Note that $a = a ba = \sum_{g \in G} a b_g a$. It follows that $b = b_e \in S_e$. Hence, $S_e$ is von Neumann regular.
%\end{proof}

We show that all graded von Neumann regular rings are nearly epsilon-strongly graded (cf. Proposition \ref{prop:s-unital1}).

\begin{proposition}
Let $S=\bigoplus_{g \in G} S_g$ be a $G$-graded ring. If $S$ is graded von Neumann regular, then $S$ is nearly epsilon-strongly $G$-graded. 
\label{prop:nearly}
\end{proposition}
\begin{proof}
Take an arbitrary $g \in G$ and $s \in S_g$. To prove that $S$ is nearly epsilon-strongly $G$-graded, we need to show that there exist elements $\epsilon_g(s) \in S_g S_{g^{-1}}$ and $\epsilon'_g(s) \in S_{g^{-1}} S_g$ such that $\epsilon_g(s) s =  s = s \epsilon_g'(s)$. Since $S$ is graded von Neumann regular, it follows by Proposition \ref{prop:ny} that there is some $b \in S_{g^{-1}}$ such that $s = s b s$. Then, $\epsilon_g(s) := sb \in S_g S_{g^{-1}}$ and $\epsilon_g'(s) := bs \in S_{g^{-1}}S_g$ satisfy the requirement. Hence, $S$ is nearly epsilon-strongly $G$-graded.
\end{proof}

\begin{remark}
%\begin{enumerate}[(a)]
%\begin{item}
Theorem \ref{thm:hazrat_lpa} together with Proposition \ref{prop:nearly} implies that every Leavitt path algebra over a field is nearly epsilon-strongly $\mathbb{Z}$-graded. The stronger statement that  every Leavitt path algebras over a general unital ring is nearly epsilon-strongly $\mathbb{Z}$-graded has been proved by Nystedt and Öinert \cite[Thm. 1.3]{nystedt2017epsilon}.
%\end{item}
%\begin{item}
%The element-wise condition of graded von Neumann regularity implies the component-wise condition of symmetrical gradings, i.e. $S_g S_{g^{-1}} S_g = S_g$ for every $g \in G$. 

%If $S$ is graded von Neumann regular, then $S$ is nearly epsilon-strongly $G$-graded. In particular, $S$ is symmetrically $G$-graded (see Remark \ref{rem:0}). However, there are symmetrically $G$-graded rings which are not graded von Neumann regular (see Example \ref{ex:3}). 
%\end{item}
%\end{enumerate}
\end{remark}
The following definition was introduced by Clark, Exel and Pardo \cite{clark2018generalized} in the context of Steinberg algebras: 

\begin{definition}(\cite[Def. 4.5]{clark2018generalized})
Let $S=\bigoplus_{g \in G}S_g$ be a $G$-graded ring. If $S_g = S_g S_{g^{-1}} S_g$ for every $g \in G$, then we say that $S$ is \emph{symmetrically $G$-graded}.
\label{def:sym}
\end{definition}
Moreover, Nystedt and Öinert \cite[Prop. 3.3]{nystedt2017epsilon}  proved that every nearly epsilon-strongly $G$-graded ring is symmetrically $G$-graded. In conclusion, the following relationship holds between the mentioned classes of group graded rings:

\begin{remark}
The following implications hold for an arbitrary $G$-grading $\{ S_g \}_{g \in G}$ of $S$:
\begin{equation*}
\text{unital strong } \implies \text{epsilon-strong }  \implies \text{nearly epsilon-strong } \implies \text{symmetrical}
\end{equation*}
\label{rem:1}
\end{remark}
\vspace{-1.5em}

The following corollary is a direct consequence of Proposition \ref{prop:nearly} and Remark \ref{rem:1}:
\begin{corollary}
Let $S=\bigoplus_{g \in G} S_g$ be a $G$-graded ring. If $S$ is graded von Neumann regular, then $S$ is symmetrically $G$-graded. 
\label{cor:sym}
\end{corollary}
\begin{remark}
By Corollary \ref{cor:sym}, the elementwise condition of graded von Neumann regularity (cf. Proposition \ref{prop:ny}) implies the componentwise condition of symmetrical gradings, that is $S_g = S_g S_{g^{-1}} S_g$ for every $g \in G$. However, the reverse implication does not hold in general (see Example \ref{ex:1}(b)).
\end{remark}

%
%The following lemma gives another, more technical, necessary condition: 
%
%\begin{lemma}
%Let $S=\bigoplus_{g \in G} S_g$ be graded von Neumann regular. For every $g \in G$ and $x \in S_g$, the left $S_e$-ideal $S_{g^{-1}} x$ is finitely generated.
%\label{lem:fg1}
%\end{lemma}
%\begin{proof}
%Take an arbitrary $g \in G$ and $x \in S_g$. Since $S$ is graded von Neumann regular, there exists some $y \in S_{g^{-1}}$ such that $x=xyx$. Then, $S_e (yx) = (S_ey)x \subseteq (S_eS_{g^{-1}}) x \subseteq S_{g^{-1}} x$. Conversely, let $zx \in S_{g^{-1}}x$ for some $z \in S_{g^{-1}}$. Then, $$zx = z(xyx) = (zx) yx \subseteq S_{g^{-1}} S_g (yx) \subseteq S_e (yx).$$ Thus, $S_{g^{-1}} x = S_e (yx)$ and in particular, the left $S_e$-ideal $S_{g^{-1}}x$ is finitely generated.
%\end{proof}
%
%\begin{example}
%Let $R$ be a ring and consider the Laurent polynomial ring $R[x, x^{-1}]$ equipped with the standard $\mathbb{Z}$-grading, i.e $S_i := R x^i$ for every $i \in \mathbb{Z}$. It is straightforward to prove that this $\mathbb{Z}$-grading is strong. Take an arbitrary integer $n \in \mathbb{Z}$ and element $r x^n \in S_n= Rx^n$. Then, we have that $S_{-n} (r x^n) = (R x^{-n}) (r x^n) = R r$. That is, the left $S_e$-ideal $S_{-n} x$ is finitely generated for every element $x \in S_n$. 
%Now, assume that $R$ is not von Neumann regular. It follows by Theorem \ref{thm:hazrat} that $R[x, x^{-1}]$ is not graded von Neumann regular. Hence, $R[x, x^{-1}]$ satisfies the technical condition in Lemma \ref{lem:fg1} but is not graded von Neumann regular.
%\label{ex:3b}
%\end{example}

Before proving our characterization, we need the following lemma:

\begin{lemma}
Let $S=\bigoplus_{g \in G} S_g$ be a nearly epsilon-strongly $G$-graded ring and suppose that $S_e$ is von Neumann regular. Then, for every $g \in G$ and $x \in S_g$, the left $S_e$-ideal $S_{g^{-1}}x$ is generated by an idempotent in $S_e$.
\label{lem:johan}
\end{lemma}
\begin{proof}
Take an arbitrary $g \in G$ and $x \in S_g$. Since $S$ is nearly epsilon-strongly $G$-graded, there exists some $\epsilon_g(x) \in S_g S_{g^{-1}}$ such that $\epsilon_g(x) x = x$. We can write $\epsilon_g(x) = \sum_{i=1}^k a_i b_i$ for some elements $a_1, a_2, \dots, a_k \in S_g$ and $b_1, b_2, \dots, b_k \in S_{g^{-1}}$. Let $c_i := b_i x \in S_{g^{-1}} x$ for $i \in \{ 1,2,\dots, k\}$. We claim that $S_{g^{-1}} x = S_e c_1 + S_e c_2 + \dots  + S_e c_k$. Indeed, let $s x \in S_{g^{-1}} x$ be an arbitrary element. Then, $$ sx = s (\epsilon_g(x) x) = s \Big (\sum_{i=1}^k a_i b_i \Big)x = \sum_{i=1}^k (s a_i) (b_i x) = \sum_{i=1}^k (s a_i) c_i.$$ Since $s a_i \in S_{g^{-1}} S_g  \subseteq S_e$, it follows that $S_{g^{-1}} x$ is finitely generated by $\{ c_1, c_2, \dots, c_k \}$ as a left $S_e$-ideal. Moreover, since $S_e$ is von Neumann regular, Proposition \ref{prop:s-unital1} implies that $S_e$ is s-unital. By Proposition \ref{prop:von1}, we have that $S_{g^{-1}}x$ is generated by an idempotent in $S_e$.
\end{proof}

%The necessary conditions are, in fact, also sufficient.
The following proposition generalizes Yahya's proof (see \cite[Thm. 3]{yahya1997note}) of Theorem \ref{thm:hazrat}:

\begin{proposition}
Let $S=\bigoplus_{g \in G} S_g$ be a nearly epsilon-strongly $G$-graded ring. If $S_e$ is von Neumann regular, then $S$ is graded von Neumann regular.
\label{prop:main}
\end{proposition}
\begin{proof}
Suppose that $S_e$ is von Neumann regular. Take an arbitrary $g \in G$ and $0 \ne x \in S_g$. By Proposition \ref{prop:ny}, we need to show that there exists some $r \in S_{g^ {-1}}$ such that $x= x rx$. By Lemma \ref{lem:johan}, there is some idempotent $y \in S_e$ such that $S_{g^{-1}} x = S_e y$. Note that $y = y^2 \in S_e y = S_{g^{-1}} x$. Hence, there is some $r \in S_{g^{-1}}$ such that $y = r x$. Also note that,
\begin{equation}
S_g S_{g^{-1}} x = S_g (S_{g^{-1}} x) = S_g (S_e y) = S_g S_e y \subseteq S_g y.
\label{eq:a1}
\end{equation}
Since $S$ is assumed to be nearly epsilon-strongly $G$-graded, there exists some $\epsilon_g(x) \in S_g S_{g^{-1}}$ such that $\epsilon_g(x) x = x$. Now, using (\ref{eq:a1}), we have that $x = \epsilon_g(x) x \in S_g S_{g^{-1}} x \subseteq S_g y$, and hence there exists some $x' \in S_g$ such that $x=x'y$. But then $xy = (x'y)y = x' (yy) = x'y = x$. Thus, $x= xy = x r x$. Hence, $S$ is graded von Neumann regular. 
\end{proof}

%\begin{remark}
%The assumption that the $S_g$'s are finitely generated as left modules over $S_e$ is not a necessary condition for $S$ to be graded von Neumann regular (cf. Example \ref{ex:not_fg}).
%\end{remark}

Now we can prove our characterization of graded von Neumann regular rings:

\begin{proof}[Proof of Theorem 1.2]
Let $S=\bigoplus_{g \in G}S_g$ be a $G$-graded ring. Suppose that $S$ is graded von Neumann regular. Then Proposition \ref{prop:nearly} and Lemma \ref{lem:1} establish that $S$ is nearly epsilon-strongly $G$-graded and that $S_e$ is von Neumann regular, respectively. Conversely, suppose that $S$ is nearly epsilon-strongly $G$-graded and that $S_e$ is von Neumann regular. Then Proposition \ref{prop:main} implies that $S$ is graded von Neumann regular. 
\end{proof}

%\begin{remark}
%Neither (a), (b) nor (c) in Theorem 1.2 is alone sufficient to imply that $S$ is graded von Neumann regular. Indeed, Example \ref{ex:3a} shows that (b) does not suffice. Example \ref{ex:3b} shows that (c) alone is not enough. Moreover, in Example \ref{ex:3}, we will see an example of a nearly epsilon-strongly graded ring that is not graded von Neumann regular. However, it is possible  that (a) and/or (b) implies (c). 
%\end{remark}

Since epsilon-strongly $G$-graded rings are nearly epsilon-strongly $G$-graded (see Remark \ref{rem:1}), the following result is a consequence of Theorem \ref{thm:char}.

\begin{corollary}
Let $S=\bigoplus_{g \in G} S_g$ be an epsilon-strongly $G$-graded ring. Then $S$ is graded von Neumann regular if and only if $S_e$ is von Neumann regular. 
\label{cor:epsilon}
\end{corollary}
%\begin{proof}
%Assume that $S=\bigoplus_{g \in G} S_g$ is an epsilon-strongly $G$-graded ring. By \cite[Prop. 7(iv)]{nystedt2016epsilon}, $S_g$ is finitely generated as a left $S_e$-module for every $g \in G$. Fix an arbitrary $g \in G$ and $x \in S_g$. Say that $S_{g^{-1}}$ is generated by $w_1, w_2, \dots, w_n$. Then, $S_{g^{-1}} x = \sum_{i=1}^n S_e w_i x = \sum_{i=1}^n S_e z_i$ where we put $z_i := w_i x$. Hence, it follows that $S_{g^{-1}} x$ is finitely generated as a left $S_e$-module. Moreover, since an epsilon-strongly $G$-graded ring is nearly epsilon-strongly $G$-graded (see Remark \ref{rem:1}(a)), the assertions (a) and (b) in Theorem \ref{thm:char} are satisfied. Thus, by Theorem \ref{thm:char}, $S$ is graded von Neumann regular if and only if $S_e$ is von Neumann regular.
%\end{proof}

\begin{remark}
The above result is a generalization of Theorem \ref{thm:hazrat}. Indeed, by applying Corollary \ref{cor:epsilon} to unital strongly group graded rings, which by Remark \ref{rem:1} are epsilon-strongly graded, we immediately recover Theorem \ref{thm:hazrat}.
%Let $S=\bigoplus_{g \in G} S_g$ be a unital strongly $G$-graded ring. In particular (), $S$ is epsilon-strongly $G$-graded. By Corollary \ref{cor:epsilon}, $S$ is graded von Neumann regular if and only if $S_e$ is von Neumann regular. In other words, Corollary \ref{cor:epsilon} generalizes Theorem \ref{thm:hazrat}. 
\end{remark}

%It is not clear to the author if the statement in Theorem \ref{thm:main} holds for nearly epsilon-strongly graded rings. 

%
%\begin{corollary}
%Let $S=\bigoplus_{g \in G} S_g$ be an epsilon-strongly $G$-graded ring. Assume that $S_e$ is von Neumann regular, then any graded right (left) $S$-module is flat.
%\end{corollary}
%
%\begin{corollary}
%Let $S=\bigoplus_{g \in \mathbb{Z}} S_g$ be an epsilon-strongly $\mathbb{Z}$-graded ring.
%If $S_e$ is von Neumann regular, then $S$ is semi-primitive.
%\end{corollary}

\section{Proof of Theorem \ref{thm:big}}
\label{sec:applications}

In this section, we prove that a Leavitt path algebra $L_R(E)$ is graded von Neumann regular if and only if $R$ is von Neumann regular (see Theorem \ref{thm:big}).

Let $E=(E^0, E^1, s, r)$ be a directed graph consisting of a vertex set $E^0$, an edge set $E^1$ and maps $s \colon E^1 \to E^0$ and $r \colon E^1 \to E^0$ specifying the source vertex $s(f)$ respectively range vertex $r(f)$ for each edge $f \in E^1$. Note that we allow the \emph{null-graph} which has no vertices and no edges. If $E$ is not the null-graph, i.e. $E^0 \ne \emptyset$, then we write $E \ne \emptyset$. A \emph{sink} is a vertex $v \in E^0$ such that $s^{-1}(v) = \emptyset$. An \emph{infinite emitter} is a vertex $v \in E^0$ such that $s^{-1}(v)$ is an infinite set. A vertex is called \emph{regular} if it is neither a sink nor an infinite emitter. The set of regular vertices of $E$ is denoted by $\text{Reg}(E)$.

For technical reasons we will consider a generalization of Leavitt path algebras introduced by Ara and Goodearl \cite{ara2012leavitt}:

\begin{definition}
Let $R$ be a unital ring and let $E=(E^0, E^1, s, r)$ be a directed graph. Moreover, let $X$ be any subset of $\text{Reg}(E)$. The \emph{Cohn path algebra relative to $X$}, denoted by $C_R^X(E)$, is the free $R$-algebra generated by the symbols, $$\{ v \mid v \in E^0 \} \cup \{ f \mid f \in E^1 \} \cup \{ f^* \mid f \in E^1 \}, $$ subject to the following relations:
\begin{enumerate}[(i)]
\begin{item}
$v v' = \delta_{v, v'} v$ for all $v, v' \in E^0$,
\end{item}
\begin{item}
$s(f)f=f=fr(f)$ for all $f \in E^1$,
\end{item}
\begin{item}
$r(f)f^* = f^* = f^*s(f)$ for all $f \in E^1$,
\end{item}
\begin{item}
$f^* f' = \delta_{f, f'} r(f)$ for all $f, f' \in E^1$,
\end{item}
\begin{item}
$v = \sum_{f \in s^{-1}(v)} f f^*$ for all $v \in X$.
\end{item}
\end{enumerate}
\label{def:cohn}
We let $R$ commute with the generators.

\smallskip

\noindent
Taking $X=\text{Reg}(E)$, we obtain the \emph{Leavitt path algebra of $E$ over $R$}. In other words, we have that $L_R(E) = C_R^{\text{Reg}(E)}(E)$. 
\end{definition}

Recall that a \emph{path} is a sequence of edges $\alpha=f_1 f_2 \dots f_n$ such that $r(f_i)=s(f_{i+1})$ for $1 \leq i \leq n-1$. The \emph{length} of $\alpha$ is equal to $n$ and we write $\text{len}(\alpha) = n$. We also write $s(\alpha) = s(f_1)$ and $r(\alpha)=r(f_n)$. By convention, a vertex $v \in E^0$ is considered to be a path of length $0$. Moreover, there is an anti-graded involution on $C_R^X(E)$ defined by $f \mapsto f^*$ for every $f \in E^1$ and $v \mapsto v^*=v$ for every $v \in E^0$. This involution extends to paths by putting $\alpha^* = f_n^* f_{n-1}^* \dots f_1^*$. The element $\alpha \in C_R^X(E)$ is called a \emph{real path} and $\alpha^* \in C_R^X(E)$ is called a \emph{ghost path}.  Let $\text{Path}(E)$ be the set of paths in $E$. In particular, $\text{Path}(E)$ includes the vertices of $E$ since they are considered zero length paths. Elements of the form $\alpha \beta^* \in C_R^X(E)$ for $\alpha, \beta \in \text{Path}(E)$ are called \emph{monomials}. It can be shown that any element of $C_R^X(E)$ can be written as a finite sum $\sum r_i \alpha_i \beta_i^*$ where $r_i \in R$ and $\alpha_i, \beta_i \in \text{Path}(E)$. Furthermore, there is a natural $\mathbb{Z}$-grading of relative Cohn path algebras given by, 
\begin{equation}
(C_R^X(E))_i  = \text{Span}_R \{ \alpha \beta^* \mid \alpha, \beta \in \text{Path}(E), \text{len}(\alpha)-\text{len}(\beta) = i \},
\label{eq:1bb}
\end{equation} for every $i \in \mathbb{Z}$. This $\mathbb{Z}$-grading is called the \emph{canonical $\mathbb{Z}$-grading} of $C_R^X(E)$. %For more details about Cohn path algebras and Leavitt path algebras, we refer the reader to the monograph by Abrams, Ara and Siles Molina \cite{abrams2017leavitt}.

The canonical $\mathbb{Z}$-grading of Leavitt path algebras was studied by Hazrat \cite{hazrat2014leavitt}. Among other results, he proved that if $E$ is a finite graph, then $L_R(E)$ is strongly $\mathbb{Z}$-graded if and only if $E$ has no sinks (see \cite[Thm. 3.15]{hazrat2014leavitt}). Nystedt and Öinert established that $L_R(E)$ is epsilon-strongly $\mathbb{Z}$-graded if $E$ is finite (\cite[Thm. 1.2]{nystedt2017epsilon}) and that $L_R(E)$ is nearly epsilon-strongly $\mathbb{Z}$-graded for any graph $E$ (see \cite[Thm. 1.3]{nystedt2017epsilon}). For more details about Cohn path algebras and Leavitt path algebras, we refer the reader to the monograph by Abrams, Ara and Siles Molina \cite{abrams2017leavitt}.

\smallskip

We now consider graded von Neumann regular Leavitt path algebras. The following example shows that graded von Neumann regularity of $L_R(E)$ is dependent on $R$. Recall that any ring $S$ is trivially $G$-graded by any group $G$ by putting $S_e = S$ and $S_g = \{ 0 \}$ for  every $g \ne e$. 

\begin{example}
Let $R$ be a unital ring and consider the following directed graph:
\begin{displaymath}
A_1 : \qquad
	\xymatrix{
	\bullet_{v} 
	}
\end{displaymath}

\noindent
Note that since $A_1$ does not contain any edges, we have that the canonical $\mathbb{Z}$-grading (cf. (\ref{eq:1bb})) of $L_R(A_1)$ is given by $L_R(A_1)_i =  \{ 0 \}$ for $i \ne 0$ and $L_R(A_1)_0 = \text{Span}_R \{ v \}.$ 
\begin{enumerate}[(a)]
\begin{item}
The $\mathbb{Z}$-graded ring $L_R(A_1)$ is graded isomorphic to the coefficient ring $R$ equipped with the trivial $\mathbb{Z}$-grading via the map defined by $r \mapsto r v$ for every $r \in R$. With this grading, every element is homogeneous.  Hence, $L_R(A_1)$ is graded von Neumann regular if and only if $R$ is von Neumann regular.
\end{item}
\begin{item}
Furthermore, since $A_1$ is a finite graph, it follows by \cite[Thm. 1.2]{nystedt2017epsilon} that $L_R(A_1)$ is epsilon-strongly $\mathbb{Z}$-graded and therefore, in particular, symmetrically $\mathbb{Z}$-graded (see Remark \ref{rem:1}). If $R$ is not von Neumann regular, then $L_R(A_1)$ is not graded von Neumann regular by (a). However, $L_R(A_1)$ is symmetrically $\mathbb{Z}$-graded. This shows that not all symmetrically graded rings are graded von Neumann regular (cf. Corollary \ref{cor:sym}).
\end{item}
\end{enumerate}
\label{ex:1}
\end{example}

%Furthermore, we can now give an example of a graded ring that is symmetrically graded but not graded von Neumann regular. 

%\begin{example}
%Let $R$ be a unital ring that is not von Neumann regular and consider the graph $A_1$ given in Example \ref{ex:1}. Since $A_1$ is a finite graph, it follows by \cite[Thm. 28]{nystedt2017epsilon} that $L_R(A_1)$ is epsilon-strongly $\mathbb{Z}$-graded and therefore, in particular, nearly epsilon-strongly graded and symmetrically $\mathbb{Z}$-graded (see Remark \ref{rem:1}).  On the other hand, by the conclusion of Example \ref{ex:1},  it follows that $L_R(A_1)$ is not graded von Neumann regular. Hence, $L_R(A_1)$ is an example of a ring that is symmetrically $\mathbb{Z}$-graded but not graded von Neumann regular.
%\label{ex:3}
%\end{example}

%In fact, the seemingly special conclusion of Example \ref{ex:1} holds for any Leavitt path algebra except for the Leavitt path algebra associated to the null-graph:
%In fact, the seemingly special conclusion of Example \ref{ex:1}(a) holds for any Leavitt path algebra (see Theorem \ref{thm:big}).

Let us now briefly discuss our method for proving Theorem \ref{thm:big}. Let $R$ be a unital ring and let $E$ be a directed graph. By \cite[Thm. 1.3]{nystedt2017epsilon}, $L_R(E)$ is nearly epsilon-strongly $\mathbb{Z}$-graded. It follows from Theorem \ref{thm:char} that $L_R(E)$ is graded von Neumann regular if and only if $(L_R(E))_0$ is von Neumann regular. If $E$ is a finite graph, then we can explicitly describe $(L_R(E))_0$ which allows us to lift von Neumann regularity from $R$ to $(L_R(E))_0$ (see Theorem \ref{thm:dn}). However, if $E$ is not finite, then this approach does not seem to work. 

Instead, the proof of Theorem \ref{thm:big}  proceeds as follows: We first prove the theorem in the special case of finite graphs using Corollary \ref{cor:epsilon} (see Corollary \ref{cor:finite}). Secondly, we reduce the general case to the finite case by writing any Leavitt path algebra as a direct limit of Leavitt path algebras of finite graphs (see Proposition \ref{prop:finite_graphs}). This latter reduction step is similar to the technique used by Hazrat \cite[Steps (II)-(IV)]{hazrat2014leavitt} to establish Theorem \ref{thm:hazrat_lpa}.

%\begin{remark}
%
%\end{remark}

%\begin{remark}
%Nystedt and Öinert \cite{nystedt2017epsilon} proved that any Leavitt path algebra is nearly epsilon-strongly graded. An alternative approach to prove Theorem \ref{thm:big} might begin with Theorem \ref{thm:char} to establish that $L_R(E)$ is graded von Neumann regular if and only if $(L_R(E))_0$ is von Neumann regular for any (possibly infinite) graph $E$. However, our proof that $R$ is von Neumann regular if and only if $(L_R(E))_0$ is von Neumann regular only seems to work when $E$ is finite (see Theorem \ref{thm:dn}). For this reason, we will not pursue this approach.
%\end{remark}

\subsection{Finite graphs}

%In this section, we will prove that the theorem holds for finite graphs. In the next section, we will extend it to any graph $E$.

%Let $R$ be a unital ring and let $E$ be a finite directed graph. Recall that Nystedt and Öinert proved that the Leavitt path algebra $L_R(E)$ is epsilon-strongly $\mathbb{Z}$-graded (see \cite[Thm. 28]{nystedt2017epsilon}).
% Next,

Let $R$ be a unital ring and let $E$ be a finite graph. The principal component $(L_R(E))_0$ is a unital subring of $L_R(E)$ with multiplicative identity element $1_{(L_R(E))_0} = \sum_{v \in E^0} v$ (see e.g. \cite[Lem. 1.2.12(iv)]{abrams2017leavitt}). We begin by characterizing when $(L_R(E))_0$ is von Neumann regular. This will follow from a more general structure theorem. For Leavitt path algebras over fields, this result was showed by Ara, Moreno and Pardo (see the proof of \cite[Thm. 5.3]{ara2007nonstable}). However, their proof generalizes to Leavitt path algebras over unital rings in a straightforward manner. Define a filtration of $(L_R(E))_0$ as follows. For $n \geq 0 $, put,
\begin{equation*}
D_n = \text{Span}_R \{ \alpha \beta^* \mid \text{len}(\alpha) = \text{len}(\beta) \leq n \}.
\end{equation*} It is straightforward to show that $D_n$ is an $R$-subalgebra of $(L_R(E))_0$. For $v \in E^0$ and $n > 0$ let $P(n,v)$ denote the set of paths $\gamma$ with $\text{len}(\gamma)=n$ and $r(\gamma)=v$. Let $\text{Sink}(E)$ denote the set of sinks in $E$. Moreover, recall that a matricial ring is a finite product of full matrix rings.

\begin{theorem}(\cite[Cor. 2.1.16]{abrams2017leavitt})
Let $R$ be a unital ring and let $E$ be a finite directed graph. For a non-negative integer $n$, let $M_n(R)$ denote the full $n \times n$-matrix ring. Then, 
\begin{equation*}
(L_R(E))_0 = \bigcup_{n \geq 0} D_n.
\end{equation*}
Moreover, we have that,
\begin{equation*}
D_0  \cong \prod_{v \in E^0} R ,
\label{eq:u5}
\end{equation*} 
\begin{equation*}
D_n \cong \prod_{\substack{0 \leq i \leq n-1 \\ v \in \text{Sink}(E)}} M_{| P(i, v) |}(R) \times \prod_{v \in E^0} M_{ | P(n, v) | }(R),
\label{eq:u6}
\end{equation*}
as $R$-algebras. In particular, $(L_R(E))_0$ is a direct limit of matricial rings over $R$ (as an object in the category of unital rings).
\label{thm:dn}
%\todo{I don't think we need a full proof here?}
\end{theorem}
We can now establish the following lemma:

\begin{lemma}
Let $R$ be a unital ring and let $E$ be a finite directed graph.  If $R$ is von Neumann regular, then $(L_R(E))_0$ is von Neumann regular.
\label{lem:leavitt_principal1}
\end{lemma}
\begin{proof}
%Since $E$ is finite, we have that $(L_R(E))_0$ is unital with multiplicative identity element $1 = \sum_{v \in E^0} v$. Furthermore, the elements $v \in (L_R(E))_0$ for $v \in E^0$ are by definition pairwise orthogonal idempotents. 

%Assume that $R$ is von Neumann regular. Take an arbitrary el

%The principal component $(L_R(E))_0$ is the direct limit of matricial rings over $R$ (see proof of \cite[Thm. 5.3]{ara2007nonstable} and \cite[Cor. 2.1.16]{abrams2017leavitt}). 

The principal component $(L_R(E))_0$ is the direct limit of matricial rings over $R$ by Theorem \ref{thm:dn}. A full matrix ring over $R$ is von Neumann regular if and only if $R$ is von Neumann regular. Moreover, recall that unital von Neumann regular rings are closed under direct limits (see \cite[Prop. 5.2.14]{berrick2000categories}). It follows that $(L_R(E))_0$ is von Neumann regular if $R$ von Neumann is regular. 
\end{proof}
For the converse statement, we do not need the assumption that $E$ is a finite graph, but $E$ cannot be the null-graph (see Remark \ref{ex:2}).

\begin{lemma}
Let $R$ be a unital ring and let $E \ne \emptyset$ be a directed graph. If $(L_R(E))_0$ is von Neumann regular, then $R$ is von Neumann regular.
\label{lem:leavitt_principal2}
\end{lemma}
\begin{proof}
Suppose that $(L_R(E))_0$ is von Neumann  regular and fix an arbitrary vertex $v_0 \in E^0$, whose existence is guaranteed by the assumption that $E \ne \emptyset$. Note that $R \xhookrightarrow{} (L_R(E))_0$ via the map $r \mapsto r v_0$. Take an arbitrary element $0 \ne t \in R$. By the assumption there is some $x \in (L_R(E))_0$ such that $t v_0 = (t v_0) x (t v_0)$. Let $x = \sum_i r_i \alpha_i \beta_i^*$ for some $\alpha_i, \beta_i \in \text{Path}(E)$ satisfying $\text{len}(\alpha_i)=\text{len}(\beta_i)$, $r(\alpha_i)=r(\beta_i)$ and $r_i \in R$ for each index $i$. Then,
\begin{align}
t v_0 & = (  t v_0  ) \Big (\sum_i r_i \alpha_i \beta_i^*\Big ) ( t v_0 ) = \sum_j t r_j  t \alpha_j \beta_j^* ,
\label{eq:1}
\end{align}
where the sum goes over all indices $j$ such that $s(\alpha_j)=s(\beta_j)=v_0$. Consider the finite set $M = \{ \alpha_j \}$ of $\alpha_j$'s appearing in right hand sum of (\ref{eq:1}). Let $\alpha_m$ be a fixed path of maximal length appearing in $M$.  Multiplying both sides of (\ref{eq:1}) with $\alpha_m^*$ from the left yields, 
\begin{equation}
\alpha_m^* ( t v_0 ) = \alpha_m^* \Big ( \sum_j t r_j  t \alpha_j \beta_j^*  \Big ) = \sum_j t r_j t \alpha_m^* \alpha_j \beta_j^* = \sum_j t r_j t (\alpha_m^* \alpha_j) \beta_j^*.
\label{eq:3}
\end{equation}
Since $s(\alpha_m)=v_0$, we have that $r(\alpha_m^*)=v_0$ and hence $\alpha_m^* (t v_0) = t \alpha_m^* \ne 0$. Recall (see \cite[Lem. 1.2.12(i)]{abrams2017leavitt}) that for any paths $\delta, \mu \in \text{Path}(E)$ we have that,
\begin{equation}
\delta^* \mu = \begin{cases}
\kappa & \text{ if } \mu = \delta \kappa \text{ for some } \kappa \\
\sigma^* & \text{ if } \delta = \mu \sigma \text{ for some } \sigma \\
0 & \text{otherwise.}
\end{cases}
\label{eq:2}
\end{equation}
Using (\ref{eq:2}) and the assumption that $\alpha_m$ is of maximal length it follows that if $\alpha_m = \alpha_j \alpha_j'$ then $\alpha_m^* \alpha_j = (\alpha_j')^*$. Otherwise, i.e. if $\alpha_j$ is not an initial segment of $\alpha_m$, we have that $\alpha_m^* \alpha_j = 0$. Hence, (\ref{eq:3}) simplifies to,
\begin{equation*}
0 \ne t \alpha_m^* = \sum_k t r_k t (\beta_k')^*,
\end{equation*}
for some paths $\beta_k'$. Since ghost paths are $R$-linearly independent (see \cite[Prop. 4.9]{tomforde2011leavitt}), it follows that $t \alpha_m^* = t r_k t \alpha_m^*$ for some index $k$. Hence, $t = t r_k t$ and thus $R$ is von Neumann regular.
%Either $E$ contains a sink or $E$ does not contain any sinks. First suppose that $v_0 \in E^0$ is a sink and multiply (\ref{eq:1}) on both sides with $v_0$. We obtain that $r v_0 = \sum_j r r_j r \alpha_j \beta_j^*$ where the sum is over the indices $j$ such that $s(\alpha_j)=s(\beta_j) = v_0$. But since $v_0$ is a sink, we must have that $\alpha_j \beta_j^* = v_0$ and thus $r v_0 = r r_0 r v_0$. Hence, $r = r r_0 r$. 	
%TODO: other direction
%TODO: Check this again.
\end{proof}

Hazrat proved that if $K$ is a field and $E$ a finite directed graph, then $L_K(E)$ is a graded von Neumann regular ring (see \cite[pg. 5]{hazrat2014leavitt}). His proof is based on a reduction to corner skew Laurent polynomial rings. Here, we obtain a generalization of his result:

\begin{corollary}
Let $R$ be a unital ring and let $E \ne \emptyset$ be a finite directed graph. Then $L_R(E)$ is graded von Neumann regular if and only if $R$ is von Neumann regular.
\label{cor:finite}
\end{corollary}
\begin{proof}
By \cite[Thm. 1.2]{nystedt2017epsilon}, $L_R(E)$ is epsilon-strongly $\mathbb{Z}$-graded. It follows from Corollary \ref{cor:epsilon} that $L_R(E)$ is graded von Neumann regular if and only if $(L_R(E))_0$ is von Neumann regular. By Lemma \ref{lem:leavitt_principal1} and Lemma \ref{lem:leavitt_principal2}, $(L_R(E))_0$ is  von Neumann regular if and only if $R$ is von Neumann regular. The statement now follows. 
\end{proof}

\begin{remark}
The null-graph is a degenerate case that needs to be excluded in Corollary \ref{cor:finite}. Let $R$ be a unital ring and let $\emptyset$ be the null-graph, i.e. the graph without any vertices or edges. In this case, $L_R(\emptyset)$ is the zero ring which is trivially von Neumann regular and hence also graded von Neumann regular. In other words, the Leavitt path algebra $L_R(\emptyset)$ is graded von Neumann regular for any unital ring $R$. 
\label{ex:2}
\end{remark}

For the rest of this section, we will reduce the general case of Theorem \ref{thm:big} to the finite case dealt with in Corollary \ref{cor:finite}. For Leavitt path algebras over a field, it is known that any Leavitt path algebra is the direct limit of Leavitt path algebras associated to finite graphs (see \cite[Cor. 1.6.11]{abrams2017leavitt}). We will show that this property generalizes to Leavitt path algebras over general unital rings. 

\subsection{Cohn path algebras as Leavitt path algebras}

A surprising but well-known result is that any relative Cohn path algebra with coefficients in a field is graded isomorphic to a Leavitt path algebra over the same field. We recall the construction from \cite[Def. 1.5.16]{abrams2017leavitt}. Consider a pair $(E,X)$ where $X \subseteq \text{Reg}(E)$. Define a new graph $E(X)$ in the following way. Let $Y := \text{Reg}(E) \setminus X$ and add new vertices $Y' = \{ v' \mid v \in Y \}$. The new graph $E(X)$ is given by:
\begin{equation*}
(E(X))^0 = E^0 \sqcup Y' \qquad \text{ and } \qquad (E(X))^1 = E^1 \sqcup \{ e' \mid r(e) \in Y \},
\end{equation*}
where $e'$ is a new edge going from $s(e)$ to the new vertex $r(e)'$.

\begin{proposition}(cf. \cite[Thm. 1.5.18]{abrams2017leavitt}) Let $R$ be a unital ring and let $E$ be a directed graph. Let $E(X)$ denote the directed graph defined above. Then, $$ C_R^X(E) \cong_{\text{gr}} L_R(E(X)).$$
\label{prop:cpa_as_lpa}
\end{proposition}
\begin{proof}
\vspace{-1em}
Let $Y=\text{Reg}(E) \setminus X$. Define a map $\phi \colon C_R^X(E) \to L_R(E(X))$. For $v \in E^0$, let $\phi(v) = v +v'$ if $v \in Y$ and $\phi(v)=v$ otherwise. For $f \in E^1$, let $\phi(f)=f + f'$ if $r(f) \in Y$ and $\phi(f)=f$ otherwise. Moreover, let $\phi(f^*)=\phi(f)^*$ for all $f \in E^1$. Using the same arguments as in \cite[Thm. 1.5.18]{abrams2017leavitt}, it follows that $\phi$ is a well-defined ring isomorphism. Furthermore, it is clear from the definition that $\phi$ is $\mathbb{Z}$-graded. 
\end{proof}

%Using the same reduction technique as Hazrat (see (II)-(IV) in \cite{hazrat2014leavitt}), we obtain the following:
%
%\begin{theorem}(cf. \cite[Thm. 9]{hazrat2014leavitt})
%Let $R$ be a unital ring and let $E$ be a directed graph. Then the Leavitt path algebra $L_R(E)$ is graded von Neumann regular if and only if the ring $R$ is graded von Neumann regular. 
%\label{thm:lpa}
%\end{theorem}
%\begin{proof}
%TODO
%\end{proof}

%Recall that the ring of integers $\mathbb{Z}$ is not regular. Hence, Corollary \ref{cor:finite} implies the following result as a special case:
%
%\begin{corollary}
%Let $E$ be a finite directed graph. Then $L_\mathbb{Z}(E)$ is not graded regular. 
%\label{cor:not_graded_regular}
%\end{corollary}
%
%\begin{remark}
%Let $E$ be any finite directed graph. By Corollary \ref{cor:not_graded_regular}, $L_\mathbb{Z}(E)$ is not graded regular while $L_\mathbb{C}(E)$ is graded regular. This result gives a separation between Leavitt path algebras with coefficients in the integers respectively the complex numbers.
%\label{rem:1}
%\end{remark}
%
%\begin{remark}
% Assuming there is a complete `Rosetta stone' between graph $C^*$-algebras and Leavitt path algebras, one would assume that $L_R(E_2) \cong L_R(E_2^-)$ as $*$-algebras. In the case $R=\mathbb{Z}$, this was proven to \emph{not} hold by Johansen and Sørensen \cite{johansen2016cuntz}. Their result actualizes the question of how important the coefficient ring is to the structure of the associated Leavitt path algebra. 
%\end{remark}

\subsection{The Cohn path algebra functor}
Let $E$ and $F$ be directed graphs. A \emph{graph homomorphism} $\phi \colon E \to F$ is a pair of maps $(\phi^0 \colon E^0 \to F^0, \phi^1 \colon E^1 \to F^1)$ satisfying the conditions $s(\phi^1(f)) = \phi^0(s(f))$ and $r(\phi^1(f)) = \phi^0(r(f))$ for every $f \in E^1$.  
Recall (see \cite[Def. 1.6.2]{abrams2017leavitt}) that the category $\mathscr{G}$ consists of objects of the form $(E,X)$ where $E$ is a directed graph and $X \subseteq \text{Reg}(E)$ is a subset of regular vertices. If $(F,Y), (E,X) \in \text{Ob}(\mathscr{G})$, then $\psi = (\psi^0, \psi^1)$ is a morphism in $\mathscr{G}$ if the following conditions are satisfied:
\begin{enumerate}[(a)]
\begin{item}
$\psi \colon F \to E$ is a graph homomorphism such that $\psi^0$ and $\psi^1$ are injective;
\end{item}
\begin{item}
$\psi^0(Y) \subseteq X$;
\end{item}
\begin{item}
For every $v \in Y$, the restriction $\psi^1 \colon s^{-1}_F(v) \to s^{-1}_E(\psi^0(v))$ is a bijection.
\end{item}
\end{enumerate}

The category $\mathscr{G}$ has arbitrary direct limits (see \cite[Prop. 1.6.4]{abrams2017leavitt}).  We will define a functor $C_R$ from $\mathscr{G}$ to $\mathbb{Z} \mhyphen \textrm{RING}$ for any unital ring $R$.

%Furthermore, any object $(E,X) \in \text{Ob}(\mathscr{G})$ is the direct limit of a direct system of the form $\{ (F_i, X_i) \mid i \in I \}$ where each $F_i$ are finite graphs (see \cite[Lem. 1.6.9]{abrams2017leavitt}).

\begin{lemma}(cf. \cite[Lem. 1.6.3]{abrams2017leavitt})
Let $\psi \colon (F,Y) \to (E,X)$ be a morphism in $\mathscr{G}$. Then there is an induced $\mathbb{Z}$-graded ring homomorphism $\bar \psi \colon C_R^Y(F) \to C_R^X(E)$.
\end{lemma}
\begin{proof}
Put $\bar \psi(v) = \psi^0(v)$, $\bar \psi(f) = \psi^1(f)$ and $\bar \psi(f^*) = \psi^1(f)^*$ for all $v \in F^0, f \in F^1$. We show that $\bar \psi$ respects the relations (i)-(v) in Definition \ref{def:cohn}. 

(i): Take arbitrary vertices $u,v \in F^0$ such that $u \neq v$. Then, since $\psi^0$ is injective by (a), it follows that $\bar \psi(u) \bar \psi(v) = \psi^0 (u) \psi^0 (v) = 0.$ Furthermore, $\bar \psi(u) \bar \psi(u) = \psi^0 (u) \psi^0 (u) = \psi^0(u) = \bar \psi(u)$. This shows that $\bar \psi$ preserves (i). 

(ii)-(iii): The assumption in (a) that $\psi$ is a graph homomorphism implies that (ii) and (iii) are preserved. 

(iv): Follows by injectivity of $\psi^1$ similarly to (i).

(v): Let $v \in F^0$ with $s^{-1}_F(v) \ne \emptyset$. By (c), $\psi^1$ maps $s^{-1}_F(v)$ bijectively onto $s^{-1}_E(\psi^0(v))$. Hence, $$\bar \psi(v) = \psi^0(v) = \sum_{f \in s^{-1}_E(\psi^0(v))} f f^* = \sum_{f \in s^{-1}_F(v)} \psi^1(f) \psi^1(f)^* = \sum_{f \in s^{-1}_F(v)} \bar \psi(f) \bar \psi(f).$$

\noindent
Thus, $\bar \psi$ extends to a well-defined ring homomorphism $C_R^Y(F) \to C_R^X(E)$. Furthermore, it follows directly from the definition that $\bar \psi$ is $\mathbb{Z}$-graded.
\end{proof}

The following functor has previously only been considered in the case of coefficients in a field. But in fact, the properties we need also hold true for arbitrary coefficient rings.

\begin{definition}
Let $R$ be a unital ring. Define the \emph{Cohn path algebra functor} by,
\begin{align*}
C_R \colon \mathscr{G} &\to \mathbb{Z} \mhyphen \textrm{RING}, \\
(E, X) &\mapsto C_R^X(E) \\
\psi &\mapsto \bar \psi,
\end{align*}
for all objects $(E,X) \in \text{Ob}(\mathscr{G})$ and morphisms $\psi$ in $\mathscr{G}$. 
\end{definition}

\begin{lemma}
(cf. \cite[Prop. 1.6.4]{abrams2017leavitt}) The functor $C_R$ preserves direct limits.
\label{lem:cont}
\end{lemma}
\begin{proof}
Let $((E_i, X_i), (\phi_{ji})_{i,j \in I, j\geq i})$ be a directed system in $\mathscr{G}$ with direct limit $((E,X), \psi_i)$. We  show that $(C_R^X(E), \bar \psi_i)$ is the direct limit of the directed system $(C_R^{X_i}(E_i), \bar \phi_{ji})$. 
Let $A$ be a $\mathbb{Z}$-graded ring and $\gamma_i \colon C_R^{X_i}(E_i) \to A$ be a family of compatible morphisms.  We need to show that there is a $\mathbb{Z}$-graded ring homomorphism $\gamma \colon C_R^X(E) \to A$ making the following diagram commute: 
\begin{displaymath}
	\xymatrix@C+1em@R+1em{ 
	C_R^{X_i}(E_i) \ar^{\bar \phi_{ij}}[rr] \ar^{\bar \psi_i}[dr] \ar@/_1em/_{\gamma_i}[ddr] & & C_R^{X_j}(E_j) \ar_{\bar \psi_j}[dl] \ar@/^1em/^{\gamma_j}[ddl] \\
	& C_R^X(E) \ar@{-->}^\gamma[d] & \\
	& A &
	}
\end{displaymath}
Define $\gamma \colon C_R^X(E) \to A$ by, $$ \gamma(\psi_i^s(\alpha)) = \gamma_i(\alpha), \quad \gamma(\psi_i^s(\alpha)^*) = \gamma_i(\alpha^*), $$
for all $\alpha \in E_i^s,\, i \in I, $ and $ s \in \{ 0, 1 \}.$ It remains to show that this gives a well-defined $\mathbb{Z}$-graded ring homomorphism. We show that $\gamma$ preserves the relations (i)-(v) in Definition \ref{def:cohn}. 

(i): Let $u \in E^0$. Then there is some $i \in I$ and $u_0 \in E_i^0$ satisfying $\psi_i(u_0)=u$. Hence, by definition, $$\gamma(u)^2 = \gamma(\psi_i^0(u_0))^2 = \gamma_i(u_0)^2 = \gamma_i(u_0^2) = \gamma_i(u_0) = \gamma(u).$$ Let $u \neq v \in E^0$ and take some $i \in I$ such that $u_0, v_0 \in E_i^0$ and $\psi_i^0(u_0)=u, \psi_i^0(v_0) = v$. Since $u_0 \neq v_0$, it follows that $\gamma(u) \gamma(v) = \gamma_i(u_0) \gamma_i(v_0) = \gamma_i(u_0 v_0) = \gamma_i(0) = 0.$

(ii): Let $f \in E^1$. Then there is some $i \in I$ such that there are $f_0 \in E_i^1$ and $v_0 \in E_i^0$ satisfying $\psi_i^1(f_0)=f$ and $\psi_i^0(v_0) = s(f)$. By assumption (a), $\psi_i$ is an injective graph homomorphism, which implies that $v_0 = s(f_0)$.  Then, 
\begin{align*}
\gamma(s(f)) \gamma(f) &= \gamma(\psi_i^0(v_0)) \gamma(\psi_i^1(f_0)) = \gamma_i(v_0) \gamma_i(f_0) = \gamma_i(v_0 f_0) = \gamma_i(s(f_0) f_0) = \\& = \gamma_i(f_0) = \gamma(\psi_i^1(f_0)) = \gamma(f).
\end{align*}

(iii): Analogous to (ii).

(iv): Let $f \in E^1$ and take $i \in I$ such that there is some $f_0 \in E_i^1$ and $v_0 \in E_i^0$ satisfying $\psi_i^1(f_0)=f$ and $\psi_i^0(v_0) = r(f)$. By assumption (a), $\psi_i$ is an injective graph homomorphism, which implies that $v_0 = r(f_0)$. Then, 
\begin{align*}
\gamma(f^*) \gamma(f) & = \gamma(\psi_i^1(f_0^*)) \gamma(\psi_i^1(f_0)) = \gamma_i(f_0^*) \gamma_i (f_0) = \gamma_i(f_0^* f_0) = \gamma_i(r(f_0) = \\ &= \gamma_i(v_0) = \gamma(\psi_i^0(v_0)) = \gamma(v).
\end{align*}

(v): Let $v \in E^0$ and take $i \in I$ such that there is some $v_0 \in E_i^0$ satisfying $\psi_i^0(v_0) = v$. Since, $\psi_i^1$ maps $s_{E_i}^{-1}(v_0)$ bijectively onto $s_{E}^{-1}(\psi_i^0(v_0))=s_E^{-1}(v)$, it follows that, 
\begin{align*}
\gamma(v) &= \gamma(\psi_i^0(v_0)) = \gamma_i(v_0) = \sum_{f \in s_{E_i}^{-1}(v_0)} \gamma_i(f) \gamma_i(f^*)  = \sum_{f \in s_{E_i}^{-1}(v_0)}  \gamma(\psi_i^1(f)) \gamma(\psi_i^1(f^*)) =\\ & = \sum_{f' \in S_E^{-1}(v)} \gamma(f') \gamma((f')^*).
\end{align*}
Thus, $\gamma$ is a well-defined ring homomorphism. Moreover, it follows directly from the definition that it is $\mathbb{Z}$-graded. Since $\gamma_i$ and $\gamma \circ \bar \psi_i$ agree on the generators $E_i^0 \cup E_i^1  \cup (E_i^1)^*$  of $C_R^X(E)$, it follows that $\gamma_i = \gamma \circ \bar \psi_i$. Hence, the diagram commutes and we are done.
\end{proof}

\begin{proposition}(cf. \cite[Cor. 1.6.11]{abrams2017leavitt})
Let $R$ be a unital ring and let $E$ be a graph. Then there exists a directed system $\{ (F_i, Y_i) \mid i \in I \}$ in $\mathscr{G}$ such that the following assertions hold:
\begin{enumerate}[(a)]
\begin{item}
every $F_i$ is a finite directed graph;
\end{item}
\begin{item}
$L_R(E) \cong_{\text{gr}} \varinjlim_{i} C_R^{Y_i}(F_i) \cong_{\text{gr}} \varinjlim_{i} L_R(F_i(Y_i)).$ %In other words, $L_R(E)$
\end{item}
\end{enumerate}
In other words, $L_R(E)$ is graded isomorphic to the direct limit of Leavitt path algebras associated to the finite graphs $F_i(Y_i)$. 
\label{prop:finite_graphs}
\end{proposition}
\begin{proof}
By \cite[Lem. 1.6.9]{abrams2017leavitt}, we have that $(E, \text{Reg}(E))$ is the direct limit of some directed system $\{ (F_i, Y_i) \mid i \in I \}$ where every $F_i$ is a finite graph. Since the functor $C_R$ preserves direct limits by Lemma \ref{lem:cont}, we have that $L_R(E) \cong_{\text{gr}} \varinjlim_{i} C_R^{Y_i}(F_i)$. By Proposition \ref{prop:cpa_as_lpa}, it follows that $C_R^{Y_i}(F_i) \cong_{\text{gr}}L_R(F_i(Y_i))$ for each $i \in I$. Hence, $\varinjlim_{i} C_R^{Y_i}(F_i) \cong_{\text{gr}} \varinjlim_{i} L_R(F_i(Y_i))$.
\end{proof}

The following proposition establishes the difficult direction of Theorem \ref{thm:big}:

\begin{proposition}
Let $R$ be a unital ring and let $E$ be a directed graph. If $R$ is von Neumann regular, then $L_R(E)$ is graded von Neumann regular.
\label{prop:lift_general}
\end{proposition}
\begin{proof}
%Since $R$ is von Neumann regular, Corollary \ref{cor:finite} implies that $L_R(F)$ is graded von Neumann regular for any finite non-trivial graph $F$. Moreover, Example \ref{ex:3} shows that the Leavitt path algebra of the trivial graph over $R$ is graded von Neumann regular. Hence, $L_R(F)$ is graded von Neumann regular for any finite graph $F$. By Proposition \ref{prop:finite_graphs}, $L_R(E)$ is graded isomorphic to a direct limit of Leavitt paths of finite graphs. Thus, by Lemma \ref{lem:direct_lim}, we have that $L_R(E)$ is graded von Neumann regular. 
If $E=\emptyset$ is the null-graph, then $L_R(E)$ is trivially graded von Neumann regular (see Remark \ref{ex:2}). Next, suppose that $E \neq \emptyset$. Since $R$ is von Neumann regular, Corollary \ref{cor:finite} implies that $L_R(F)$ is graded von Neumann regular for any finite  graph $F$. By Proposition \ref{prop:finite_graphs}, $L_R(E)$ is graded isomorphic to a direct limit of Leavitt path algebras associated to finite graphs. Thus, by Lemma \ref{lem:direct_lim}, we have that $L_R(E)$ is graded von Neumann regular. 
\end{proof}
We are now ready to give a complete proof of Theorem \ref{thm:big}.
\begin{proof}[Proof of Theorem \ref{thm:big}]
Let $E \ne \emptyset$ be a directed graph. First suppose that $R$ is von Neumann regular. Then Proposition \ref{prop:lift_general} implies that $L_R(E)$ is graded  von Neumann regular. 

Conversely, suppose that $L_R(E)$ is graded von Neumann regular. Then, by Lemma \ref{lem:1}, it follows that $(L_R(E))_0$ is von Neumann regular. Moreover, by Lemma \ref{lem:leavitt_principal2}, this implies that $R$ is  von Neumann regular. 
\end{proof}

\begin{remark}
Let $E$ be an arbitrary graph. Since $\mathbb{C}$ is von Neumann regular it follows from Theorem \ref{thm:big} that $L_\mathbb{C}(E)$ is graded von Neumann regular. On the other hand, since $\mathbb{Z}$ is not von Neumann regular, we conclude that $L_\mathbb{Z}(E)$ is not graded von Neumann regular. Hence, Theorem \ref{thm:big}, in particular, algebraically differentiate Leavitt path algebras with coefficients in $\mathbb{C}$ respectively $\mathbb{Z}$. The author feels that this differentiaton is especially interesting considering the strange behaviour of Leavitt path algebras with coefficients in $\mathbb{Z}$ observed by Johansen and S{\o}rensen \cite{johansen2016cuntz}.
\label{rem:3}
\end{remark}

\section{Semiprime and semiprimitive Leavitt path algebras} 
\label{sec:semiprime}
In this section, we apply our results to obtain sufficient conditions for a Leavitt path algebra over a unital ring to be semiprimitive and semiprime. 

Abrams and Aranda Pino showed that if $K$ is a field and $E$ is a graph, then the Leavitt path algebra $L_K(E)$ is both semiprime and semiprimitive (see \cite[Prop. 6.1-6.3]{abrams2008leavitt}). However, Leavitt path algebras over non-field rings are not always semiprime nor semiprimitive. Indeed, for the graph $A_1$ in Example \ref{ex:1}, we have that $L_R(A_1) \cong R$ for any unital ring $R$. Hence, $L_R(A_1)$ is semiprime/semiprimitive if and only if $R$ is semiprime/semiprimitive. 

Let $R$ be a ring. Recall that an ideal $I$ of $R$ is called \emph{semiprime} if $aRa \subseteq I$ implies that $a \in I$ for every $a \in R$. A ring is called \emph{semiprime} if its zero ideal is semiprime. Moreover, recall that $R$ is called \emph{semiprimitive} if the Jacobson radical $J(R)=0$. Let $S$ be a $G$-graded ring and recall that $S$ is said to have \emph{homogeneous local units} if there is a set of local units $E$ for $S$ consisting of homogeneous idempotents. The following lemma by Abrams and Arando Pino generalizes Bergman's famous result that if $S$ is a unital $\mathbb{Z}$-graded ring, then $J(S)$ is a graded ideal of $S$ (see \cite[Cor. A.I.7.15]{nastasescu1982graded}):

\begin{lemma}
(\cite[Lem. 6.2]{abrams2008leavitt}) Let $S$ be a $\mathbb{Z}$-graded ring. Suppose that $S$ has a set of homogeneous local units. Then $J(S)$ is a graded ideal of $S$. 
\label{lem:bergman1}
\end{lemma}
\begin{remark}
Let $S$ be a unital $\mathbb{Z}$-graded ring. Then $E= \{ 1_S \}$ is a set of local units for $S$. Moreover, note that $1_S \in S_e$ is a homogeneous element. Hence, it follows from Lemma \ref{lem:bergman1} that $J(S)$ is a graded ideal of $S$. 
\label{rem:2}
\end{remark}

We have Leavitt path algebras in mind when we state the following result, but we will also need it in the next section.

\begin{proposition}(cf. \cite[Prop. 2(4)]{hazrat2014leavitt}, \cite[Prop. 6.3]{abrams2008leavitt})
Let $S$ be a $\mathbb{Z}$-graded ring. Suppose that $S$ is graded von Neumann regular and that $S$ has a set of homogeneous local units. Then $S$ is semiprimitive and semiprime.
\label{prop:semi1}
\end{proposition}
\begin{proof}
By Lemma \ref{lem:bergman1}, $J(S)$ is a graded ideal. Let $x \in J(S)$ be a homogeneous element and consider the graded left ideal $Sx \subseteq J(S)$. It follows from Proposition \ref{prop:hazrat_char} that there is some idempotent $f$ such that $S x = S f$. Recall that the Jacobian radical does not contain any non-zero idempotents. But $f=f^2 \in S f = S x \subseteq J(S)$, which implies that $f = 0$. Since $S$ has  a set of local units, it follows that $x \in Sx = Sf = 0$ and hence $x=0$. Thus, $J(S)=0$ and hence $S$ is semiprimitive.

By \cite[Prop. 2(2)]{hazrat2014leavitt}, every graded ideal of a non-unital graded von Neumann regular ring is semiprime. Thus, the zero ideal of $S$ is semiprime and hence $S$ is semiprime.
%This proves that $S$ is semiprimitive. The conclusion of the proposition now follows since a semiprimitive ring is semiprime (see e.g. \cite[Expl. 10.17(e)]{lam2001first}).
\end{proof}

%\begin{proposition}(cf. \cite[Prop. 6.3]{abrams2008leavitt})
%Let $R$ be a unital ring and let $E$ be a directed graph. If $R$ is von Neumann regular, then $L_R(E)$ is semiprimitive and semiprime.
%\label{prop:semiprimitive}
%\end{proposition}
%\begin{proof}
%Assume that $R$ is von Neumann regular. Note that by Proposition \ref{prop:lift_general}, $L_R(E)$ is graded von Neumann regular. By \cite[Prop. 1]{hazrat2014leavitt}, this implies that every finitely generated graded one-sided ideal of $L_R(E)$ is generated by a homogeneous idempotent.  Bergman famously proved that the Jacobian of a unital $\mathbb{Z}$-graded ring is a graded ideal. Abrams and Aranda Pino extended this result to include Leavitt path algebras (see \cite[Lem. 6.2]{abrams2008leavitt}). In other words, $J = J(L_R(E))$ is a graded ideal of $L_R(E)$. Let $x \in J$ be a homogeneous element. Then there is some idempotent $f$ such that $L_R(E) x = L_R(E) f$. Recall the Jacobian radical does not contain any nonzero idempotents. But, $f=f^2 \in L_R(E) f = L_R(E) x \subseteq J$, which implies that $f = 0$. Hence, $x = 0$ and thus $J=0$. This proves that $L_R(E)$ is semiprimitive. The conclusion of the proposition now follows since a semiprimitive ring is semiprime (see e.g. \cite[Expl. 10.17(e)]{lam2001first}).
%\end{proof}

Finally, we prove that Leavitt path algebras with coefficients in a von Neumann regular ring are semiprimitive and semiprime.

\begin{corollary}
Let $R$ be a unital ring and let $E$ be a directed graph. If $R$ is von Neumann regular, then $L_R(E)$ is semiprimitive and semiprime.
\label{cor:semiprimitive}
\end{corollary}
\begin{proof}
Note that $E = \{ v \mid v \in E^0 \}$ is a set of local units for $L_R(E)$ consisting of homogeneous elements. Suppose that $R$ is von Neumann regular. Then, by Theorem \ref{thm:big}, $L_R(E)$ is graded von Neumann regular. It follows from Proposition \ref{prop:semi1} that $L_R(E)$ is semiprimitive and semiprime. 
\end{proof}

\begin{remark}
Since a field is von Neumann regular, it follows that Corollary \ref{cor:semiprimitive} generalizes Abrams and Aranda Pino's result that Leavitt path algebras over fields are semiprimitive and semiprime (see \cite[Prop. 6.1-6.3]{abrams2008leavitt}).
\end{remark}

\section{More applications}
\label{sec:applications2}
In this last section, we apply our results to unital partial crossed products and corner skew Laurent polynomial rings.
Partial crossed products were introduced as a generalization of the classical crossed products (see \cite{dokuchaev2008crossed}). Among these, the \emph{unital partial crossed products} were shown to be especially well-behaved (see e.g. \cite{bagio2010crossed}). Let $R$ be a unital ring and let $G$ be a group with neutral element $e$. A \emph{unital twisted partial action of $G$ on $R$} (see \cite[pg. 2]{nystedt2016epsilon}) is a triple 
$( \{\alpha_g \}_{g \in G}, \{ D_g \}_{g \in G}, \{ w_{g,h} \}_{(g,h) \in G \times G})$
satisfying certain technical relations. To this triple, it is possible to associate an epsilon-strongly $G$-graded algebra $R \star_{\alpha}^\omega G$ called the \emph{unital partial crossed product}. The following result shows that unital partial crossed products behave similarly to classical crossed products with regards to graded von Neumann  regularity.

\begin{corollary}
Let $G$ be a group, let $R$ be a unital ring and let $R \star_{\alpha}^\omega G$ be a unital partial crossed product. Then the unital partial crossed product $R \star_{\alpha}^\omega G$ is graded von Neumann regular if and only if $R$ is von Neumann regular. 
\label{cor:partial_crossed_products}
\end{corollary}
\begin{proof}
The ring $R \star_{\alpha}^\omega G$ is epsilon-strongly $G$-graded (see \cite[pg. 2]{nystedt2016epsilon}) with principal component $R$. The statement now follows from Corollary \ref{cor:epsilon}.
\end{proof}

The general construction of fractional skew monoid rings was introduced by Ara, Gonzalez-Barroso, Goodearl and Pardo in \cite{ara2004fractional}. We consider the special case of a fractional skew monoid ring by a corner isomorphism which is also called a \emph{corner skew Laurent polynomial ring}. Let $R$ be a unital ring and let $\alpha \colon R \to eRe$ be a corner ring isomorphism where $e$ is an idempotent of $R$. The corner skew Laurent polynomial ring, denoted by $R[t_{+}, t_{-}; \alpha]$, is a unital epsilon-strongly $\mathbb{Z}$-graded ring (see \cite[Prop. 8.1]{lannstrom2019graded}). 

The following result was proved by Hazrat using direct methods. We recover it as a special case of Corollary \ref{cor:epsilon}.

\begin{corollary}[cf. {\cite[Prop. 8]{hazrat2014leavitt}}]
Let $R$ be a unital ring, let $e$ be an idempotent of $R$ and let $\phi \colon R \to e Re$ be a corner isomorphism. The corner skew Laurent polynomial ring $R[t_{+}, t_{-}, \phi]$ is graded von Neumann regular if and only if $R$ is a von Neumann regular ring.
\label{cor:corner}
\end{corollary}
\begin{proof}
By \cite[Prop. 8.1]{lannstrom2019graded}, $R[t_{+}, t_{-}, \phi]$ is epsilon-strongly $\mathbb{Z}$-graded with principal component $R$. The desired conclusion now follows from Corollary \ref{cor:epsilon}. 
\end{proof}

We end this article by given sufficient conditions for a corner skew Laurent polynomial ring to be semiprimitive and semiprime.

\begin{corollary}
Let $R$ be a unital ring, let $e$ be an idempotent of $R$ and let $\phi \colon R \to e Re$ be a corner isomorphism. If $R$ is a von Neumann regular ring, then the corner skew Laurent polynomial ring $R[t_{+}, t_{-}, \phi]$ is semiprimitive and semiprime.
\label{cor:semi2}
\end{corollary}
\begin{proof}
Since $R[t_{+}, t_{-}, \phi]$ is a unital $\mathbb{Z}$-graded ring, it follows that $E = \{ 1_R \}$ is a set of local units (see Remark \ref{rem:2}). Suppose that $R$ is von Neumann regular. Then, by Corollary \ref{cor:corner}, it follows that $R[t_{+}, t_{-}, \phi]$ is graded von Neumann regular. The conclusion now follows by Proposition \ref{prop:semi1}.
\end{proof}
 
 \section*{acknowledgement}
This research was partially supported by the Crafoord Foundation (grant no. 20170843). The author is grateful to Johan Öinert for giving a proof of Lemma \ref{lem:johan}. The author is also grateful to Stefan Wagner, Johan Öinert and Patrik Nystedt for giving feedback and comments that helped to improve this manuscript. 

%\begin{theorem}(\cite[Thm. 28]{nystedt2017epsilon})
%Let $R$ be a unital ring and let $E$ be a finite directed graph. Then the Leavitt path algebra $L_R(E)$ is epsilon-strongly $\mathbb{Z}$-graded.
%\label{thm:nystedt}
%\end{theorem}

\printbibliography

\end{document}